\documentclass[]{interact}

\usepackage{epstopdf}
\usepackage[caption=false]{subfig}
\usepackage[numbers,sort&compress]{natbib}
\usepackage[latin9]{inputenc}
\usepackage{geometry}
\usepackage{color}
\usepackage{babel}
\usepackage{mathrsfs}
\usepackage{amsmath}
\usepackage{amsthm}
\usepackage{amssymb}
\usepackage{setspace}
\usepackage[unicode=true,pdfusetitle,
 bookmarks=true,bookmarksnumbered=false,bookmarksopen=false,
 breaklinks=false,pdfborder={0 0 1},backref=false,colorlinks=true]
 {hyperref}
\hypersetup{
 pdfborderstyle=,linkcolor=blue,urlcolor=cyan,citecolor=red}
\usepackage{babel}
\usepackage{babel}
\usepackage{xpatch}
\usepackage[misc]{ifsym}

\bibpunct[, ]{[}{]}{,}{n}{,}{,}
\renewcommand\bibfont{\fontsize{10}{12}\selectfont}

\theoremstyle{plain}
\newtheorem{theorem}{Theorem}[section]
\newtheorem{lemma}[theorem]{Lemma}
\newtheorem{corollary}[theorem]{Corollary}

\theoremstyle{definition}
\newtheorem{definition}[theorem]{Definition}
\newtheorem{example}[theorem]{Example}

\theoremstyle{remark}
\newtheorem{remark}{Remark}

\begin{document}


\title{Unrestricted Douglas-Rachford algorithms for solving convex feasibility problems in Hilbert space}

\author{
\name{Kay Barshad\textsuperscript{a}\thanks{\noindent CONTACT\\ \noindent  K. Barshad, kaybarshad@technion.ac.il, kaybarshad@gmail.com\\ \noindent A. Gibali, avivg@braude.ac.il\\ \noindent S. Reich, sreich@technion.ac.il}, Aviv Gibali\textsuperscript{b} and Simeon Reich\textsuperscript{c}}
\affil{\textsuperscript{a,c}Department of Mathematics, The Technion -- Israel Institute of Technology,
32000 Haifa, Israel;\\ \textsuperscript{b}Department of Mathematics, Braude College, 2161002 Karmiel, Israel}}

\maketitle

\begin{abstract}
In this work we focus on the convex feasibility problem (CFP) in Hilbert space. A specific method in this area that has gained
a lot of interest in recent years is the Douglas-Rachford (DR) algorithm.
This algorithm was originally introduced in 1956 for solving stationary and non-stationary heat equations. Then in 1979, Lions and Mercier
adjusted and extended the algorithm with the aim of solving CFPs and even more general problems, such as finding zeros of the sum of two maximally monotone operators.
Many developments which implement various concepts concerning this algorithm have occurred during the last decade. We introduce an \color{black}unrestricted \color{black}
DR algorithm, which provides a general framework for such concepts. Using \color{black}unrestricted \color{black} products of a finite number of strongly nonexpansive
operators, we apply \color{black}this framework \color{black} to provide new iterative methods, where, \textit{inter alia}, such operators may be interlaced between
the operators used in the scheme of our \color{black}unrestricted \color{black} DR algorithm.
\end{abstract}


\section{\label{sec1}\color{black}Preliminaries} \color{black}
\global\long\def\theenumi{\roman{enumi}}%

Suppose that $\mathcal{H}$ is a real Hilbert space with inner product
$\langle\cdot,\cdot\rangle$ and let $\Vert\cdot\Vert$ be the norm
induced by $\langle\cdot,\cdot\rangle$. We denote by $Id:\mathcal{H}\rightarrow\mathcal{H}$ the identity
operator on $\mathcal{H}$, that is, for all $x\in\mathcal{H}$, $Idx=x$. For an operator $T:\mathcal{H}\rightarrow\mathcal{H}$, the fixed points set of $T$ is defined as $\mathrm{Fix(\mathit{T})}:=\left\{x\in \mathcal{H} \mid Tx=x\right\}$.
For each nonempty and convex set $C\subseteq\mathcal{H}$, we denote
by $P_{C}$ the (unique) metric projection onto $C$, the existence
of which is guaranteed if $C$ is, in addition, closed. The expressions
$x_{n}\rightharpoonup x$ and $x_{n}\rightarrow x$ denote, respectively,
the weak and strong convergence to $x$ of a sequence $\left\{ x_{n}\right\} _{n=0}^{\infty}$
in $\mathscr{\left(\mathcal{H},\mathscr{\Vert\cdot\Vert}\right)}$.
Throughout this paper, $\mathbb{N}$ denotes the set of natural numbers
starting from $0$, and for any two integers $m$ and $n$, with $m\le n$,
we denote by $\left\{ m,\dots,n\right\} $ the set of all integers
between $m$ and $n$ (including $m$ and $n)$.

Now we recall the following types of algorithmic operators. For more
information on such operators, see, for example, \cite{key-11}.
\begin{definition}
Let $T:\mathcal{H}\rightarrow\mathscr{\mathcal{H}}$ be an operator
and let $\lambda\in\left[0,2\right]$. The operator $T_{\lambda}:\mathcal{H}\rightarrow\mathscr{\mathcal{H}}$
defined by $T_{\lambda}:=\left(1-\lambda\right)\mathrm{Id}+\lambda T$
is called a $\lambda$-$relaxation$ of the operator $T$. The operator
$T_{2}$ is called the \textit{reflection} of the operator $T$.
\end{definition}
\begin{remark}
\label{Relaxation has the same Fix} Clearly, $\mathrm{Fix}(T)=\mathrm{Fix}(T_{\lambda})$
for every operator $T:\mathcal{H}\rightarrow\mathcal{H}$ and every
$\lambda\in\left(0,2\right]$.
\end{remark}
\begin{definition}
An operator $T:\mathcal{H}\rightarrow\mathcal{H}$ is said to be \textit{nonexpansive}
if
\[
\left\Vert Tx-Ty\right\Vert \leq\left\Vert x-y\right\Vert \;\forall x,y\in\mathcal{H}.
\]
For $\lambda\in\left[0,2\right]$, an operator $T:\mathcal{H}\rightarrow\mathcal{H}$
is said to be $\lambda$-\textit{relaxed nonexpansive} if $T$ is
a $\lambda$-relaxation of a nonexpansive operator $U$, that is,
$T=U_{\lambda}$.
\end{definition}
\begin{remark}
\label{ComposNE}Clearly, a composition of nonexpansive operators
is nonexpansive.
\end{remark}
\begin{definition}
An operator $T:\mathcal{H}\rightarrow\mathcal{H}$ is said to be \textit{quasi-nonexpansive}
if $\mathrm{Fix}\left(T\right)\not=\emptyset$ and
\[
\left\Vert Tx-z\right\Vert \le\left\Vert x-z\right\Vert \;\forall (x,z)\in\mathcal{H}\times\mathrm{Fix}\left(T\right).
\]
\end{definition}
\begin{remark}
\label{NE is QNE}It is clear that a nonexpansive operator with
a fixed point is, in particular, quasi-nonexpansive.
\end{remark}
\begin{definition}
An operator $T:\mathcal{H}\rightarrow\mathcal{H}$ is called \textit{firmly
nonexpansive} if
\[
\left\langle Tx-Ty,x-y\right\rangle \ge\left\Vert Tx-Ty\right\Vert ^{2}\;\forall x,y\in\mathcal{H}.
\]
For $\lambda\in\left[0,2\right]$, an operator $T:\mathcal{H}\rightarrow\mathcal{H}$
is called $\lambda$-\textit{relaxed firmly nonexpansive} if $T$
is a $\lambda$-relaxation of a firmly nonexpansive operator $U$,
that is, $T=U_{\lambda}$.
\end{definition}
\begin{theorem}
\label{FNE-equiv.cond}Let $T:\mathcal{H}\rightarrow\mathcal{H}$
be an operator. The following conditions are equivalent.
\begin{enumerate}
\item $T$ is firmly nonexpansive.
\item $T_{\lambda}$ is nonexpansive for each $\lambda\in\left[0,2\right]$.
\item There exists a nonexpansive operator $N:\mathcal{H\rightarrow\mathcal{H}}$
such that $T=2^{-1}\left(Id+N\right)$.
\end{enumerate}
\end{theorem}
\begin{proof}
See Theorem 2.2.10 in \cite{key-11}.
\end{proof}
\begin{definition}
For $\lambda\in\left(0,1\right)$, we say that an operator $T:\mathcal{H}\rightarrow\mathcal{H}$
is \textit{$\lambda$-averaged }if it is a $\lambda$-relaxed nonexpansive
operator. We say that an operator $T:\mathcal{H}\rightarrow\mathcal{H}$
is \textit{averaged }if it is $\lambda$-averaged for some $\lambda\in\left(0,1\right)$.
\end{definition}
\begin{corollary}
\label{2.2.17}Let $\lambda\in\left(0,2\right)$. An operator $T:\mathcal{H}\rightarrow\mathcal{H}$
is $\lambda$-relaxed firmly nonexpansive if and only if it is $(2^{-1}\lambda)$-averaged.
\end{corollary}
\begin{proof}
See Corollary 2.2.17 in \cite{key-11}.
\end{proof}
\begin{theorem}
\label{thm:2.2.42}Let $n$ be a positive integer. For each $i=1,\dots,n$,
let $\lambda_{i}\in\left(0,2\right)$ and let $T_{i}:\mathcal{H}\rightarrow\mathcal{H}$
be $\lambda_{i}$-relaxed firmly nonexpansive. Then $T:=T_{n}T_{n-1}\cdots T_{1}$
is a $\lambda$-relaxed firmly nonexpansive operator for some $\lambda\in\left(0,2\right)$.
In addition, if $\cap_{i=1}^{n}\mathrm{Fix}(T_{i})\not=\emptyset$,
then $\mathrm{Fix}(T)=\cap_{i=1}^{n}\mathrm{Fix}(T_{i})$.
\end{theorem}
\begin{proof}
See Theorem 2.2.42 and Theorem 2.1.26 in \cite{key-11}.
\end{proof}
The following corollary is an immediate consequence of Theorem \ref{thm:2.2.42}
and Corollary \ref{2.2.17}.
\begin{corollary}
\label{compos.averaged}Let $n$ be a positive integer. For each $i=1,\dots,n$,
let $\lambda_{i}\in\left(0,1\right)$ and let $T_{i}:\mathcal{H}\rightarrow\mathcal{H}$
be $\lambda_{i}$-averaged. Then $T:=T_{n}T_{n-1}\cdots T_{1}$ is
a $\lambda$-averaged operator for some $\lambda\in\left(0,1\right)$.
In addition, if $\cap_{i=1}^{n}\mathrm{Fix}(T_{i})\not=\emptyset$,
then $\mathrm{Fix}(T)=\cap_{i=1}^{n}\mathrm{Fix}(T_{i})$.
\end{corollary}
\begin{example}
\label{thm:2.2.21}Let $C\subset\mathcal{H}$ be a nonempty, closed
and convex set, and let $P_{C}:\mathscr{\mathcal{H}\rightarrow\mathcal{H}}$
be the metric projection onto $C$. Then the operator $P_{C}$ is
firmly nonexpansive and hence averaged. For the proof, see Theorem
2.2.21 in \cite{key-11}.
\end{example}
\begin{definition}
\label{SNE}We write that an operator $T:\mathcal{H}\rightarrow\mathcal{H}$
satisfies \textit{Condition $\left(S\right)$ }if for each pair of
sequences $\left\{ x_{n}\right\} _{n=0}^{\infty}$ and $\left\{ y_{n}\right\} _{n=0}^{\infty}$
in $\mathcal{H}$,
\vspace{-.2cm}
\[
\left.\begin{array}{l}
\left\{ x_{n}-y_{n}\right\} _{n=0}^{\infty}\,\,\mathrm{is}\,\mathrm{\,bounded}\\
\left\Vert x_{n}-y_{n}\right\Vert -\left\Vert Tx_{n}-Ty_{n}\right\Vert \rightarrow0
\end{array}\right\} \Longrightarrow x_{n}-y_{n}-\left(Tx_{n}-Ty_{n}\right)\rightarrow0.
\]

The operator $T$ is called \textit{strongly nonexpansive }if it is
nonexpansive and satisfies Condition $\left(S\right)$.
\end{definition}
\begin{theorem}
\label{2.3.4}Let $T:\mathcal{H}\rightarrow\mathcal{H}$ be a firmly
nonexpansive operator and let $\lambda\in\left(0,2\right)$. Then
the $\lambda$-relaxation $T_{\lambda}$ of $T$ is strongly nonexpansive.
\end{theorem}
\begin{proof}
See Theorem 2.3.4 in \cite{key-11}.
\end{proof}
The following corollary is an immediate consequence of Theorem \ref{2.3.4}
and Corollary \ref{2.2.17}.
\begin{corollary}
\label{averaged is strongly non expansive}Let $T:\mathcal{H}\rightarrow\mathcal{H}$
be an averaged operator. Then $T$ is strongly nonexpansive.
\end{corollary}
\begin{definition}
An operator $T:\mathcal{H}\rightarrow\mathcal{H}$ is \textit{asymptotically
regular }if for each $x\in\mathcal{H}$, we have
\[
\left\Vert T^{n+1}x-T^{n}x\right\Vert \rightarrow0.
\]
\end{definition}
\begin{lemma}
\label{Strongly nonexpansive  is asymptotically regular} Let $T:\mathcal{H}\rightarrow\mathcal{H}$
be a strongly nonexpansive operator such that $\mathrm{Fix}\left(T\right)\not=\emptyset$.
Then $T$ is asymptotically regular.
\end{lemma}
\begin{proof}
See Lemma 3.4.9 in \cite{key-11}.
\end{proof}
\begin{definition}
\label{demi-closed operator}Let $C$ be nonempty, closed and convex
subset of $\mathcal{H}$. An operator $T:C\rightarrow\mathcal{H}$
is \textit{weakly regular} (satisfies \textit{Opial's demi-closedness
principle, $T-Id$ is demiclosed at $0$)} if for any sequence $\left\{ x_{n}\right\} _{n=0}^{\infty}\subset\mathcal{H}$
and any $x\in\mathcal{H}$, the following implication holds:
\vspace{-.2cm}
\[
\left.\begin{array}{l}
x_{n}\rightharpoonup x\\
Tx_{n}-x_{n}\rightarrow0
\end{array}\right\} \Longrightarrow x\in\mathrm{Fix}(T).
\]
\end{definition}
\begin{lemma}
\label{Nonexpansive is weakly regular} Let $T:\mathcal{H}\rightarrow\mathcal{H}$
be a nonexpansive operator. Then $T$ is weakly regular.
\end{lemma}
\begin{proof}
See Lemma 3.2.5 in \cite{key-11}.
\end{proof}
Next, we recall the following generalization of Opial's Theorem, which
we will use in the sequel.
\begin{theorem}
\label{GenOpial}Let $C\subset\mathcal{H}$ be a nonempty, closed
and convex set, and let $S:C\rightarrow\mathcal{H}$ be a weakly regular
operator such that $\mathrm{Fix}\left(S\right)\not=\emptyset$. Assume
that $x_{0}\in C$ is arbitrary and $\left\{ T_{n}\right\} _{n=0}^{\infty}$
is a sequence of quasi-nonexpansive operators $T_{n}:C\rightarrow C$
such that $\mathrm{Fix}\left(S\right)\subset\cap_{n=0}^{\infty}\mathrm{Fix}(T_{n})$.
If the sequence $\left\{ x_{n}\right\} _{n=0}^{\infty}$ generated
by the recurrence $x_{n}=T_{n-1}\left(x_{n-1}\right)$ satisfies
\[
\left\Vert Sx_{n}-x_{n}\right\Vert \rightarrow0,
\]
then it converges weakly to a point $x_{*}\in\mathrm{Fix}\left(S\right)$.
\end{theorem}
\begin{proof}
See Theorem 3.6.2 along with Remark 3.6.4 in \cite{key-11}.
\end{proof}
The following theorem is an immediate consequence of Theorem \ref{GenOpial}.
It is a well-known theorem established by Opial in \cite{key-18}.
\begin{theorem}
\label{Opial}Let $C\subset\mathcal{H}$ be a nonempty, closed and
convex set. If $T:C\rightarrow C$ is a nonexpansive and asymptotically
regular operator with $\mathrm{Fix}\left(T\right)\not=\emptyset$,
then for any $x_{0}\in C$, the sequence $\left\{ x_{n}\right\} _{n=0}^{\infty}$
generated by the recurrence $x_{n}=T\left(x_{n-1}\right)$ converges
weakly to a point $x_{*}\in\mathrm{Fix}\left(T\right)$.
\end{theorem}

\section{\label{sec2}\color{black}Introduction} \color{black}
The \textit{Convex Feasibility Problem} (CFP) in $\mathcal{H}$ is
the problem of finding a point $x_{*}\in\cap_{\alpha\in I}C_{\alpha}$,
where $\left\{ C_{\alpha}\right\} _{\alpha\in I}$ is a family of
nonempty, closed and convex subsets of $\mathcal{H}$. There are numerous
iterative methods for solving this problem in the literature (see,
for example, \cite{key-2,key-4,key-11} and references therein). The
Douglas-Rachford (DR) algorithm for a finite family of sets, originally
introduced by Douglas and Rachford in \cite{key-12} for solving the
heat equation, is one such method. It has attracted considerable
interest in the last few years. Many investigations of this algorithm
have recently been undertaken in diverse directions, as one can observe,
for example, in \cite{bcm19,ies21,key-6,key-13,key-15} and references therein.
Borwein and Tam presented in \cite{key-3} a cyclic-DR algorithm for
CFPs in which the number of sets is allowed to be greater than $2$.
In this DR algorithm the original one is applied over subsequent pairs
of sets. Artacho {\it et al.} provided in \cite{key-19} a more general
framework for the cyclic DR algorithm, where the proper employment
of certain sets and operators in this algorithm is studied. In the
present paper we offer a further extension of these ideas by introducing
an \color{black}unrestricted DR algorithm, where the sets are chosen (almost) unrestrictedly, \color{black}that is, our DR-algorithm is  (almost) unrestricted in the sense of the choice of sets which are used to obtain its outcome. \color{black}
We begin by recalling the notion of an $r$-set Douglas-Rachford ($r$-set
DR) operator, which was defined in \cite{key-15} in the following
way.
\begin{definition}
Given a finite family $\left\{ C_{i}\right\} _{i=1}^{r}$ of nonempty,
closed and convex subsets of $\mathcal{H}$, the \textit{composite
reflection} operator \color{black}$\mathcal{V}_{\left\{ C_{i}\right\} _{i=1}^{r}}:\mathcal{H}\rightarrow\mathcal{H}$ \color{black}
is defined by
\begin{equation}
\mathcal{V}_{\left\{ C_{i}\right\} _{i=1}^{r}}:=\mathcal{R}_{C_{r}}\dots\mathcal{R}_{C_{1}},\label{eq:-6}
\end{equation}
where $\mathcal{R}_{C_{i}}=2P_{C_{i}}-Id$ is the reflection with
respect to the corresponding $C_{i}$ for each $i=1,\dots,r$.
\color{black}The \textit{$r$-set DR }\color{black} operator $\mathcal{T}_{\left\{ C_{i}\right\} _{i=1}^{r}}:\mathcal{H}\rightarrow\mathcal{H}$
is defined by
\begin{equation}
\mathcal{T}_{\left\{ C_{i}\right\} _{i=1}^{r}}:=2^{-1}\left(Id+\mathcal{V}_{\left\{ C_{i}\right\} _{i=1}^{r}}\right).\label{eq:-7}
\end{equation}
\end{definition}
We recall the following property of the $r$-set DR operator established
in \cite{key-15}.
\begin{lemma}
\label{Fix r-sets operator}Let $\left\{ C_{i}\right\} _{i=1}^{r}$
be a finite family of nonempty, closed and convex subsets of $\mathcal{H}$
such that $\mathrm{int}\cap_{i=1}^{r}C_{i}\not=\emptyset$. Then
\[
\mathrm{Fix}\left(\mathcal{T}_{\left\{ C_{i}\right\} _{i=1}^{r}}\right)=\cap_{i=1}^{r}C_{i}.
\]
\end{lemma}
\begin{proof}
See Corollary 23 in \cite{key-15}.
\end{proof}
Given a finite family $\left\{ C_{i}\right\} _{i=1}^{m}$ of nonempty,
closed and convex subsets of $\mathcal{H}$, a mapping $f$ from $\mathbb{N}$
onto $\left\{ 1,\dots,m\right\} $ and an integer $r>1$, set $\left\{ C_{m,r}\left(n\right)\left(j\right)\right\} _{j=1}^{r}$
to be the finite family of subsets of $\mathcal{H}$ defined by
\begin{equation}
C_{m,r}\left(n\right)\left(j\right):=C_{f\left(\left(r-1\right)n+j-1\right)}\label{eq:-5}
\end{equation}
for each $j=1,\dots,r$ and each natural number $n$. Define a sequence
of \color{black}unrestricted \color{black} $r$-set DR operators $\left\{ S_{n}\right\} _{n=0}^{\infty}$
by
\begin{equation}
S_{n}:=\mathcal{T}_{\left\{ C_{m,r}\left(n\right)\left(j\right)\right\} _{j=1}^{r}}\label{eq:-3}
\end{equation}
\color{black}for each $n\in \mathbb{N}$. \color{black}For an arbitrary point $x_{0}\in\mathcal{H}$, we intend to analyze
the convergence of the sequence $\left\{ x_{n}\right\} _{n=0}^{\infty}$
generated by the following recurrence:
\begin{gather}
\color{black}x_{n}:=S_{n-1}\left(x_{n-1}\right)\label{eq:}.\color{black}
\end{gather}
The sequence $\left\{ x_{n}\right\} _{n=0}^{\infty}$ is the outcome
of our \color{black}unrestricted \color{black} DR algorithm.

In the setting above, let $j_{f}$ be a natural number such that
\begin{equation}
f\left(\left\{ 1,\dots,j_{f}\right\} \right)=\left\{ 1,\dots,m\right\} .\label{eq:-8}
\end{equation}
We define a \textit{composite \color{black}unrestricted \color{black} DR} operator $Q:\mathcal{H}\rightarrow\mathcal{H}$
by
\begin{equation}
Q:=S_{j_{f}}\dots S_{0}\label{eq:-4}
\end{equation}
and study the convergence properties of the sequence $\left\{ y_{n}\right\} _{n=0}^{\infty}$
generated by the following recurrence, for an arbitrary point $y_{0}\in\mathcal{H}$:
\begin{gather}
y_{n}:=Q\left(y_{n-1}\right)\label{eq:-1}
\end{gather}
This sequence is the
outcome of an alternative \color{black}unrestricted \color{black} DR algorithm.

\color{black}For a finite (possibly empty) family of operators $\left\{ T_{i}\right\} _{i=p}^{q}$ defined on $\mathcal{H}$, where $p$ and $q$ are integers, recall that for each integer $q^{\prime}\le q$, the product (composition) of $\left\{ T_{i}\right\} _{i=p}^{q^{\prime}}$, which we denote by $T_{q}\cdots T_{p}$ (it is  sometimes denoted in the literature also by $\prod_{i=p}^{q}T_{i}$) is defined by the following recurrence:
\[
T_{q^{\prime}}\cdots T_{p}:=\begin{cases}
Id & \mathrm{if}\,q^{\prime}<p\\
T_{q^{\prime}}\circ\left(T_{q^{\prime}-1}\cdots T_{p}\right) & \mathrm{if}\,q^{\prime}\ge p
\end{cases}.
\]\color{black}
Recall also that the \textit{\color{black}unrestricted \color{black} product} of a family $\left\{ T_{\alpha}\right\} _{i\in I}$
of operators determined by a mapping $h:\mathbb{N}\rightarrow I$
is the sequence of operators $\left\{ P_{n}\right\} _{n=0}^{\infty}$
defined by $P_{n}:=T_{h\left(n\right)}\dots T_{h\left(0\right)}$
for each $n\in\mathbb{N}$. For more information and applications
concerning \color{black}unrestricted \color{black} products, we refer the reader to, for example, \cite{key-10}
and \cite{key-1}.

The following theorem concerning the convergence of an \color{black}unrestricted \color{black} product
of two strongly nonexpansive operators was established in \cite{key-1}.
It provides a powerful tool for finding a common fixed point of two
such operators.
\begin{theorem}
\label{DR Theorem}Let $T_{1}$ and $T_{2}$ be two strongly nonexpansive
operators such that $\mathrm{Fix}(T_{1})\cap\mathrm{Fix}(T_{2})\not=\emptyset$.
Assume $h:\mathbb{N\rightarrow}\left\{ 1,2\right\} $ is a mapping
such that $h^{-1}\left(\left\{ i\right\} \right)$ is an infinite
set for each $i=1,2$. Then for each $x\in\mathcal{H}$, the sequence
$\left\{ P_{n}x\right\} _{n=0}^{\infty}$, where $\left\{ P_{n}\right\} _{n=0}^{\infty}$
is the \color{black}unrestricted \color{black} product of $\left\{ T_{i}\right\} _{i=1}^{2}$ determined
by the mapping $h$, converges weakly to a common fixed point of $T_{1}$
and $T_{2}$.
\end{theorem}
\begin{proof}
See Theorem 1 in \cite{key-1}.
\end{proof}
It is still an open question, whether this result can be extended
to more than two operators in the most general case. But from the
optimization theory point of view, we can actually choose how often
each operator belonging to a given family of operators should appear
in the concrete method which we apply in order to find a common fixed point
of the operators in this family. It turns out that in the case where each operator of a given
finite family appears in the product sufficiently frequently, Theorem
\ref{DR Theorem} can be extended to an arbitrary finite number of
given operators, as is implied by Theorem \ref{thm6 in DR} below.
Before formulating it, we present the following notion
of quasi-periodicity, which was used in \cite{key-1}. In similar
settings, it is sometimes called intermittency (see, for example,
\cite{key-23} and \cite{key-20}).
\begin{definition}
Assume that $M$ is a positive integer and that $h$ is a mapping
defined on $\mathbb{N}$. We say that $h$ is \textit{$M$-quasi-periodic
}if for each $n\in\mathbb{N}$, we have
\[
h\left(\left\{ n,\dots,n+M-1\right\} \right)=h\left(\mathbb{N}\right).
\]
\end{definition}
Given a finite family $\left\{ T_{i}\right\} _{i=1}^{m}$ of strongly
nonexpansive operators, a positive integer $M$ and an $M$-quasi-periodic
mapping $h$ which is onto $\left\{ 1,\dots,m\right\} $, the following
theorem, which was first proved in \cite{key-1}, establishes the
convergence of the sequence $\left\{ P_{n}x\right\} _{n=0}^{\infty}$,
where $\left\{ P_{n}\right\} _{n=0}^{\infty}$ is the \color{black}unrestricted \color{black} product
of the family $\left\{ T_{i}\right\} _{i=1}^{m}$ determined by $h$
and $x\in\mathcal{H}$ is arbitrary. Due to the high importance of
this theorem and for the convenience of the reader, we provide a detailed
proof of it in Section \ref{sec4}.
\begin{theorem}
\label{thm6 in DR}Assume that $\left\{ T_{i}\right\} _{i=1}^{m}$
is a finite family of strongly nonexpansive operators satisfying $\cap_{i=1}^{m}\mathrm{Fix}(T_{i})\not=\emptyset$,
$M$ is a positive integer and $h$ is an $M$-quasi-periodic mapping
onto $\left\{ 1,\dots,m\right\} $. Then for each $x\in\mathcal{H}$,
the sequence $\left\{ P_{n}x\right\} _{n=0}^{\infty}$ converges weakly
to a common fixed point of the operators $\left\{ T_{i}\right\} _{i=1}^{m}$,
where $\left\{ P_{n}\right\} _{n=0}^{\infty}$ is the \color{black}unrestricted \color{black} product
of $\left\{ T_{i}\right\} _{i=1}^{m}$ determined by $h$.
\end{theorem}
In order to prove the above theorem we need the following two technical
lemmata.
\begin{lemma}
\label{Lemma 1 in DR}Let $T:\mathcal{H}\rightarrow\mathcal{H}$ be
a nonexpansive mapping and $\left\{ v_{n}\right\} _{n=0}^{\infty}$
be a sequence in $\mathcal{H}$. Suppose $\left\{ v_{n}\right\} _{n=0}^{\infty}$
and $\left\{ Tv_{n}\right\} _{n=0}^{\infty}$ both converge weakly
to some $v\in\mathcal{H}$ and that $\left\Vert v_{n}\right\Vert -\left\Vert Tv_{n}\right\Vert \rightarrow0$.
Then $v\in\mathrm{Fix}(T)$.
\end{lemma}
\begin{proof}
See Lemma 1 in \cite{key-1}.
\end{proof}
\begin{lemma}
\label{Lemma 3 in DR}Let $\left\{ T_{\alpha}\right\} _{\alpha\in I}$
be a family of nonexpansive operators and $h:\mathbb{N}\rightarrow I$
be a mapping. Assume that $\left\{ P_{n_{k}}\right\} _{k=0}^{\infty}$
and $\left\{ P_{n_{k}^{\prime}}\right\} _{k=0}^{\infty}$ are two
subsequences of the \color{black}unrestricted \color{black} product $\left\{ P_{n}\right\} _{n=0}^{\infty}$
of $\left\{ T_{\alpha}\right\} _{\alpha\in I}$ determined by $h$
and let $x\in\mathcal{H}$. If $x_{1}$and $x_{2}$ are common fixed
points of the family $\left\{ T_{\alpha}\right\} _{\alpha\in I}$
above such that $P_{n_{k}}x\rightharpoonup x_{1}$ and $P_{n_{k}^{\prime}}x\rightharpoonup x_{2}$,
then $x_{1}=x_{2}$.
\end{lemma}
\begin{proof}
See Lemma 3 in \cite{key-1}.
\end{proof}
The rest of the paper is organized as follows. In Section \ref{sec3}
we formulate our main results, as well as discuss a few remarks and
examples. Section \ref{sec4} is devoted to several auxiliary results. Finally, the proofs of our main results are presented in Section \ref{sec5}.

\section{\label{sec3}Main results}
Below we state our three main theorems. We establish them in Section
\ref{sec5}.
\global\long\def\theenumi{\color{black}\roman{enumi}}%
\begin{theorem}
\label{thm1}Assume that $\left\{ C_{i}\right\} _{i=1}^{m}$ is a
finite family of nonempty, closed and convex subsets of $\mathcal{H}$,
$r>1$ is an integer and $f$ is a mapping from $\mathbb{N}$ onto
$\left\{ 1,\dots,m\right\} $. Suppose that the sequences $\left\{ S_{n}\right\} _{n=0}^{\infty}$
and $\left\{ x_{n}\right\} _{n=0}^{\infty}$ are defined, respectively,
by \eqref{eq:-3} and \eqref{eq:}, $j_{f}$ is defined by \eqref{eq:-8},
$Q$ is defined by \eqref{eq:-4} and $\left\{ y_{n}\right\} _{n=0}^{\infty}$
is defined by \eqref{eq:-1}. Then the following assertions hold:
\begin{enumerate}
\item If there exists $S:\mathcal{H}\rightarrow\mathcal{H}$, which is a
weakly regular operator with a fixed point such that $\mathrm{Fix}(S)\subset\cap_{i=1}^{m}C_{i}$,
and the sequence $\left\{ x_{n}\right\} _{n=0}^{\infty}$ satisfies
\[
\left\Vert Sx_{n}-x_{n}\right\Vert \rightarrow0,
\]
then $\left\{ x_{n}\right\} _{n=0}^{\infty}$ converges weakly to
a point $x_{*}\in\cap_{i=1}^{m}C_{i}$.
\item If $\cap_{i=1}^{m}C_{i}$ has an interior point, then the sequence
$\left\{ y_{n}\right\} _{n=0}^{\infty}$ converges weakly to a point
$y_{*}\in\cap_{i=1}^{m}C_{i}$.
\end{enumerate}
\end{theorem}
Note that part\textit{ (ii)} of Theorem \ref{thm1} above generalizes
Theorem 3.7 in \cite{key-19} (see also Remark \ref{rem:} (c) below).
Since it is often difficult to verify the existence of an operator $S$
as in part \textit{(i)}, we introduce a so-called \textit{$M$-quasi-periodic
\color{black}unrestricted \color{black} DR} algorithm in our next result, where we make a modest assumption
concerning the frequency of the sets defined by \eqref{eq:-5}. This
algorithm converges weakly if $\cap_{i=1}^{n}C_{i}\not=\emptyset$
and if, in addition, $\cap_{i=1}^{n}C_{i}$ has an interior point,
then its weak limit is a solution of the CFP defined by the family
$\left\{ C_{i}\right\} _{i=1}^{m}$.
\begin{theorem}
\label{thm2}Assume that $\left\{ C_{i}\right\} _{i=1}^{m}$ is a
finite family of nonempty, closed and convex subsets of $\mathcal{H}$,
$r>1$ and $M$ are positive integers. Let $f$\textup{ }be a mapping
from $\mathbb{N}$ onto $\left\{ 1,\dots,m\right\} $. For each natural
number $n$, let $\left\{ C_{m,r}\left(n\right)\left(j\right)\right\} _{j=1}^{r}$
be the finite family of subsets of $\mathcal{H}$ defined by \eqref{eq:-5}
so that the sequence of families $\left\{ \left\{ C_{m,r}\left(n\right)\left(j\right)\right\} _{j=1}^{r}\right\} _{n=0}^{\infty}$
is $M$-quasi-periodic (as a mapping defined on $\mathbb{N}$). Suppose
that the sequences $\left\{ S_{n}\right\} _{n=0}^{\infty}$ and $\left\{ x_{n}\right\} _{n=0}^{\infty}$
are defined, respectively, by \eqref{eq:-3} and \eqref{eq:}.
Then the following assertions hold:
\begin{enumerate}
\item If $\cap_{i=1}^{n}C_{i}\not=\emptyset$, then $\left\{ x_{n}\right\} _{n=0}^{\infty}$
converges weakly to a common fixed point of the operators $\left\{ S_{n}\right\} _{n=0}^{\infty}$.
\item If $\mathrm{int}\cap_{i=1}^{n}C_{i}\not=\emptyset$, then $\left\{ x_{n}\right\} _{n=0}^{\infty}$
converges weakly to a point $x_{*}\in\cap_{i=1}^{m}C_{i}$.
\end{enumerate}
\end{theorem}
\begin{theorem}
\label{thm3}Assume that $\left\{ C_{i}\right\} _{i=1}^{m}$ is a
finite family of nonempty, closed and convex subsets of $\mathcal{H}$
such that $\cap_{i=1}^{m}C_{i}$ has an interior point,
and that $r>1$ and $M$ are positive integers.
Suppose further that $f$\textup{ }is a mapping
from $\mathbb{N}$ onto $\left\{ 1,\dots,m\right\}$, the sequence
$\left\{ S_{n}\right\} _{n=0}^{\infty}$ is defined
 by \eqref{eq:-3}, $j_{f}$
is defined by \eqref{eq:-8} and $Q$ is defined by \eqref{eq:-4}.
Let $\left\{ T_{j}\right\} _{j=1}^{k}$ be a family of strongly nonexpansive
operators \color{black}with a common fixed point, \color{black}let $h$ be an $M$-quasi-periodic mapping onto $\left\{ 1,\dots,k\right\} $,
let $x\in\mathcal{H}$ and let $\left\{ P_{n}\right\} _{n=0}^{\infty}$
be the \color{black}unrestricted \color{black} product of the family $\left\{ T_{j}\right\} _{j=0}^{k}$
determined by $h$. Then the following assertions hold:
\begin{enumerate}
\item If for each $n\in\left\{ 1,\dots,j_{f}\right\} $, the operator $S_{n}$
is an element of $\left\{ T_{j}\right\} _{j=1}^{k}$, then $\left\{ P_{n}x\right\} _{n=0}^{\infty}$
converges weakly to a point $x_{*}\in\cap_{i=1}^{m}C_{i}$.
\item If $Q$ is an element of $\left\{ T_{j}\right\} _{j=1}^{k}$, then
$\left\{ P_{n}x\right\} _{n=0}^{\infty}$ converges weakly to a point
$x_{*}\in\cap_{i=1}^{m}C_{i}$.
\end{enumerate}
\end{theorem}
\global\long\def\theenumi{\color{black}\alph{enumi}}%
\begin{remark}\label{rem:}~
\begin{enumerate}
\item If $\mathcal{H}$ is a finite dimensional space, then the convergence in our results is strong.
\item In particular, due to Lemma \ref{Nonexpansive is weakly regular},
Theorem \ref{thm1}\textit{(i)} can be applied to the case where the
operator $S$ is a nonexpansive operator.
\item The cyclic DR algorithm, first introduced in \cite{key-19}, is obtained
by choosing in Theorem \ref{thm2} $f:\mathbb{N}\rightarrow\left\{ 1,\dots,m\right\}$
to be defined by $f\left(n\right):=n\,\,\mathrm{mod\,}\,m+1$ for
each $n\in\mathbb{N}$ and $M=m$. Here we obtain, in addition, that
it always converges to a solution of the CFP defined by the family
$\left\{ C_{i}\right\} _{i=1}^{m}$ under the assumption that $\cap_{i=1}^{n}C_{i}$
has an interior point.
\item By Theorem \ref{2.3.4}, Corollary \ref{averaged is strongly non expansive}
and Example \ref{thm:2.2.21}, Theorem \ref{thm3} can be applied
to the family of strongly nonexpansive operators $\left\{ T_{i}\right\} _{i=1}^{k}$,
which contains, in particular, $\lambda$-relaxed firmly nonexpansive
operators with the relaxation parameter lying in the interval $\left(0,2\right)$,
averaged operators and metric projections.
\end{enumerate}
\end{remark}

\section{\label{sec4}Auxiliary results}
In this section, we present several technical lemmata, which will
be used in the proofs of our main results.
\begin{lemma}
\label{lem1}For each family $\left\{ C_{i}\right\} _{i=1}^{r}$ of
nonempty, closed and convex subsets of $\mathcal{H}$, the \color{black}unrestricted \color{black}
$r$-sets-DR operator $\mathcal{T}_{\left\{ C_{i}\right\} _{i=1}^{r}}$
is averaged and hence strongly nonexpansive. In particular, if the
family $\left\{ C_{m,r}\left(n\right)\left(j\right)\right\} _{j=1}^{r}$
is defined by \eqref{eq:-5}, where $\left\{ C_{i}\right\} _{i=1}^{m}$
is a given family of nonempty, closed and convex subsets of $\mathcal{H}$,
$r>1$ is an integer and $f$ is a mapping from $N$ onto $\left\{ 1,\dots,m\right\} $,
then for each $n\in\mathbb{N}$, the operator $S_{n}$ defined by
\eqref{eq:-3} is averaged and hence strongly nonexpansive. As a result,
the composite \color{black}unrestricted \color{black} DR operator $Q$ defined by \eqref{eq:-4} is
averaged and hence strongly nonexpansive.
\end{lemma}
\begin{proof}
Let $\left\{ C_{i}\right\} _{i=1}^{r}$ be a family of nonempty, closed
and convex subsets of $\mathcal{H}$. For each $i=1,\dots,r$, the
reflection $\mathcal{R}_{C_{i}}=2P_{C_{i}}-Id$ is nonexpansive by
Example \ref{thm:2.2.21} and Theorem \ref{FNE-equiv.cond}\textit{
(i)} and \textit{(ii)}. Hence the composite reflection operator $\mathcal{V}_{\left\{ C_{i}\right\} _{i=1}^{r}}$
is nonexpansive by Remark \ref{ComposNE}. By \eqref{eq:-7} and Theorem
\ref{FNE-equiv.cond} \textit{(iii)} and \textit{(i)}, $\mathcal{T}_{\left\{ C_{i}\right\} _{i=1}^{r}}$
is firmly nonexpansive and by Corollary \ref{2.2.17}, it is averaged
and hence strongly nonexpansive by Corollary \ref{averaged is strongly non expansive}.
The rest of the statement of the lemma follows from \eqref{eq:-3},
\eqref{eq:-4}, Corollary \ref{compos.averaged}, and Corollary \ref{averaged is strongly non expansive}.
\end{proof}
\begin{lemma}
\label{lem2}Assume that $\left\{ C_{i}\right\} _{i=1}^{m}$ is a
family of nonempty, closed and convex subsets of $\mathcal{H}$, $r>1$
is an integer and $f$ is a mapping from $N$ onto $\left\{ 1,\dots,m\right\} $.
Then
\[
\cap_{i=1}^{m}C_{i}\subset\cap_{n=0}^{\infty}\mathrm{Fix}(S_{n})\subset\cap_{n=0}^{j_{f}}\mathrm{Fix}(S_{n}),
\]
where $\left\{ S_{n}\right\} _{n=0}^{\infty}$ and $j_{f}$ are
defined by, respectively, \eqref{eq:-3} and \eqref{eq:-8}. If, in
addition, $\mathrm{int}\cap_{i=1}^{m}C_{i}\not=\emptyset$, then
\[
\cap_{i=1}^{m}C_{i}=\cap_{n=0}^{\infty}\mathrm{Fix}(S_{n})=\cap_{n=0}^{j_{f}}\mathrm{Fix}(S_{n}).
\]
\end{lemma}
\begin{proof}
Clearly, by \eqref{eq:-3}, Remark \ref{Relaxation has the same Fix}
and \eqref{eq:-6}, for each $n\in\mathbb{N},$
\[
\mathrm{Fix}(S_{n})=\mathrm{Fix}\left(\mathcal{T}_{\left\{ C_{m,r}\left(n\right)\left(j\right)\right\} _{j=1}^{r}}\right)=\mathrm{Fix}\left(\mathcal{V}_{\left\{ C_{m,r}\left(n\right)\left(j\right)\right\} _{j=1}^{r}}\right)\supset\cap_{i=1}^{m}C_{i}.
\]
Hence
\[
\cap_{i=1}^{m}C_{i}\subset\cap_{n=0}^{\infty}\mathrm{Fix}(S_{n})\subset\cap_{n=0}^{j_{f}}\mathrm{Fix}(S_{n}).
\]

Assume that $\mathrm{int}\cap_{i=1}^{m}C_{i}\not=\emptyset$. Let
$z\in\cap_{n=0}^{j_{f}}\mathrm{Fix}(S_{n})$ and $p\in\left\{ 1,\dots,m\right\} $
be arbitrary. There exist $j\in\mathbb{N}$ and unique $N,k\in\mathbb{N}$
such that $f\left(j\right)=p$, $j\le j_{f}$, $k<r-1$ and $j=\left(r-1\right)N+k$.
Since $N\le j\le j_{f}$ and $k+1\le r-1$, by \eqref{eq:-3}, Lemma
\ref{Fix r-sets operator} and \eqref{eq:-5}, we have
\[
z\in\mathrm{Fix}(S_{N})=\mathrm{Fix}\left(\mathcal{T}_{\left\{ C_{m,r}\left(N\right)\left(j\right)\right\} _{j=1}^{r}}\right)=\cap_{j=1}^{r}C_{f\left(\left(r-1\right)N+j-1\right)}\subset C_{f\left(\left(r-1\right)N+k\right)}=C_{p}.
\]
Hence
\[
\cap_{n=0}^{\infty}\mathrm{Fix}(S_{n})\subset\cap_{n=0}^{j_{f}}\mathrm{Fix}(S_{n})\subset\cap_{i=1}^{m}C_{i}
\]
and the result follows.
\end{proof}

The rest of this section is devoted to the proof of Theorem \ref{thm6 in DR}.
\begin{lemma}
\label{Lem 3}Assume that $\left\{ T_{i}\right\} _{i=1}^{m}$ is a
finite family of operators, $M$ is a positive integer and $h$ is
an $M$-quasi-periodic mapping onto $\left\{ 1,\dots,m\right\} $.
Then for each strictly increasing sequence $\left\{ n_{k}\right\} _{k=0}^{\infty}$
of natural numbers such that $n_{k}\ge n_{k-1}+M$ for each positive
integer $k$, there exist a strictly increasing sequence $\left\{ n_{k}^{\prime}\right\} _{k=0}^{\infty}$
of natural numbers and a finite sequence $\left\{ l_{i}\right\} _{i=1}^{M}$
so that the set of values of $\left\{ l_{i}\right\} _{i=1}^{M}$ is
$\left\{ 1,\dots,m\right\} $ and $h\left(j\right)=l_{j-n_{n_{k}^{\prime}}}$
for each $k\in\mathbb{N}$ and each $j\in\left\{ n_{n_{k}^{\prime}}+1,\dots,n_{n_{k}^{\prime}}+M\right\} $.
As a result,
\[
P_{n_{n_{k+1}^{\prime}}}=T_{h\left(n_{n_{k+1}^{\prime}}\right)}\cdots T_{h\left(n_{n_{k}^{\prime}}+M+1\right)}T_{l_{M}}\cdots T_{l_{1}}P_{n_{n_{k}^{\prime}}},
\]
where $\left\{ P_{n}\right\} _{n=0}^{\infty}$ is the \color{black}unrestricted \color{black} product
of $\left\{ T_{i}\right\} _{i=1}^{m}$ determined by $h$.
\end{lemma}
\begin{proof}
Let $\left\{ n_{k}\right\} _{k=0}^{\infty}$ be a strictly increasing
sequence of natural numbers such that $n_{k}\ge n_{k-1}+M$ for
each positive integer $k$. Let $k\in\mathbb{N}$. Since $h$ is an $M$-quasi-periodic
mapping onto $\left\{ 1,\dots,m\right\} $, we have
$\left\{ 1,\dots,m\right\} =h\left(\left\{ n_{k}+1,\dots,n_{k}+M\right\} \right)$ for each natural number $k$.

Since the number of mappings $s:\left\{ 1,\dots,M\right\} \rightarrow\left\{ 1,\dots,m\right\} $
is finite, it follows that there exists a strictly increasing sequence
$\left\{ n_{k}^{\prime}\right\} _{k=0}^{\infty}$ of natural numbers
and a finite sequence $\left\{ l_{i}\right\} _{i=1}^{M}$ with all
the asserted properties.
\end{proof}
\begin{lemma}
\label{Lem 4}Assume that $\left\{ v_{n}\right\} _{n=0}^{\infty}$
is a bounded sequence in $\mathcal{H}$, $\left\{ T_{i}\right\} _{i=1}^{m}$
is a finite family of strongly nonexpansive operators such that the
origin is their common fixed point and $v\in\mathcal{H}$ is such that
\begin{equation}
v_{n}\rightharpoonup v\,\,\mathrm{and}\,\,\left\Vert T_{i-1}\cdots T_{1}v_{n}\right\Vert -\left\Vert T_{i}\cdots T_{1}v_{n}\right\Vert \rightarrow0\label{eq:-9}
\end{equation}
for each $i=1,\dots,m$. Then the sequence $\left\{ T_{m}\cdots T_{1}v_{n}\right\} _{n=1}^{\infty}$
converges weakly to $v$, which is also a common fixed point of the
family $\left\{ T_{i}\right\} _{i=1}^{m}$.
\end{lemma}
\begin{proof}
The proof is by induction on $m$. For $m=0$ the statement is clear.
Assume that $m>0$. By the induction hypothesis $\left\{ T_{m-1}\cdots T_{1}v_{n}\right\} _{n=0}^{\infty}$
converges weakly to $v$, which is a common fixed point of the family
$\left\{ T_{i}\right\} _{i=1}^{m-1}$. Since each $T_{i}$ is nonexpansive,
since the origin is a common fixed point of the family $\left\{ T_{i}\right\} _{i=1}^{m}$
and since the sequence $\left\{ v_{n}\right\} _{n=0}^{\infty}$ is bounded, we
see that the sequence $\left\{ T_{m-1}\cdots T_{1}v_{n}\right\} _{n=0}^{\infty}$
is also bounded. Define a sequence $\left\{ w_{n}\right\} _{n=1}^{\infty}\subset\mathcal{H}$
by $w_{n}:=0$ for all $n\in\mathbb{N}$. Since $T_{m}$ satisfies
Condition $\left(S\right)$ and
\[
\left\Vert T_{m-1}\cdots T_{1}v_{n}-w_{n}\right\Vert -\left\Vert T_{m}T_{m-1}\cdots T_{1}v_{n}-T_{m}w_{n}\right\Vert \rightarrow0,
\]
it follows that
\[
T_{m-1}\cdots T_{1}v_{n}-T_{m}\dots T_{1}v_{n}\rightharpoonup0.
\]
As a result, $T_{m}\cdots T_{1}v_{n}$ converges weakly to $v$. Combining
this with the induction hypothesis, \eqref{eq:-9} and Lemma \ref{Lemma 1 in DR}
applied to the sequence $\left\{ T_{m-1}\cdots T_{1}v_{n}\right\} _{n=0}^{\infty}$
and the operator $T_{m}$, we obtain the desired result.
\end{proof}
\begin{lemma}
\label{lem 5}Assume that $\left\{ T_{i}\right\} _{i=1}^{m}$ is a
finite family of strongly nonexpansive operators such that the origin
is their common fixed point, $M$ is a positive integer and $h$ is
an $M$-quasi-periodic mapping onto $\left\{ 1,\dots,m\right\}$.
Assume that $x\in\mathcal{H}$ and that $v$ is a weak cluster point of the
sequence $\left\{ P_{n}x\right\} _{n=0}^{\infty}$, where $\left\{ P_{n}\right\} _{n=0}^{\infty}$
is the \color{black}unrestricted \color{black} product of $\left\{ T_{i}\right\} _{i=1}^{m}$ determined
by $h$. Then $v$ is also a common fixed point of the family $\left\{ T_{i}\right\} _{i=1}^{m}$.
\end{lemma}
\begin{proof}
Clearly, there exists a strictly increasing sequence $\left\{ n_{k}\right\} _{k=0}^{\infty}$
such that $P_{n_{k}}x\rightharpoonup v$. Without any loss of generality,
we may assume that $n_{k}\ge n_{k-1}+M$ for each positive integer $k$.
By Lemma \ref{Lem 3}, \color{black} we may also assume that \color{black} there exists a finite sequence $\left\{ l_{i}\right\} _{i=1}^{M}$
so that the set of values of $\left\{ l_{i}\right\} _{i=1}^{M}$ is
$\left\{ 1,\dots,m\right\}$ and $h\left(j\right)=l_{j-n_{k}}$ for
each $k\in\mathbb{N}$ and each $j\in\left\{ n_{k}+1,\dots,n_{k}+M\right\}$.
As a result,
\begin{equation}
P_{n_{k+1}}=T_{h\left(n_{k+1}\right)}\cdots T_{h\left(n_{k+1}+M+1\right)}T_{l_{M}}\cdots T_{l_{1}}P_{n_{k}}.\label{eq:-10}
\end{equation}
Since the origin is a common fixed point of the family $\left\{ T_{i}\right\} _{i=1}^{m}$
and since each $T_{i}$ is nonexpansive, it follows that $\lim_{n\rightarrow\infty}\left\Vert P_{n}x\right\Vert$
exists. \color{black} Hence the sequence $\left\{ P_{n}x\right\} _{n=0}^{\infty}$ is bounded and by \eqref{eq:-10} \color{black}
\[
\left\Vert T_{l_{i-1}}\cdots T_{l_{1}}P_{n_{k}}x\right\Vert -\left\Vert T_{l_{i}}\cdots T_{l_{1}}P_{n_{k}}x\right\Vert \rightarrow0
\]
for each $i=1,\dots,M$ as $k\rightarrow\infty$.  Now we apply Lemma \ref{Lem 4}
to the sequence $\left\{ P_{n_{k}}x\right\} _{k=0}^{\infty}$ and
the family $\left\{ T_{l_{i}}\right\} _{i=1}^{M}$ to obtain that
the sequence $\left\{ T_{l_{M}}\cdots T_{l_{1}}P_{n_{k}}x\right\} _{k=0}^{\infty}$
converges weakly to $v$, which is a common fixed point of the family
$\left\{ T_{l_{i}}\right\} _{i=1}^{M}$ and hence is a common fixed
point of the family $\left\{ T_{i}\right\} _{i=1}^{m}$, because the
set of values of $\left\{ l_{i}\right\} _{i=1}^{M}$ is $\left\{ 1,\dots,m\right\}$.
\end{proof}
\makeatletter \renewenvironment{proof}[1][\proofname\space of Theorem \ref{thm6 in DR}] {\par\pushQED{\qed}\normalfont\topsep6\p@\@plus6\p@\relax\trivlist\item[\hskip\labelsep\bfseries#1\@addpunct{.}]\ignorespaces}{\popQED\endtrivlist\@endpefalse} \makeatother
\begin{proof}
Without any loss of generality, we may assume that the origin is a common
fixed point of the family $\left\{ T_{i}\right\} _{i=1}^{m}$. Assume
that the point $x\in\mathcal{H}$. Since the sequence $\left\{ P_{n}x\right\} _{n=0}^{\infty}$
is bounded, it has a weak cluster point $v$. Let $v_{1},v_{2}\in\mathcal{H}$
be two weak cluster points of $\left\{ P_{n}x\right\} _{n=0}^{\infty}$.
By Lemma \ref{lem 5}, they are common fixed points of the family
$\left\{ T_{i}\right\} _{i=1}^{m}$. By Lemma \ref{Lemma 3 in DR},
we have $v_{1}=v_{2}$. As a result, the sequence $\left\{ P_{n}x\right\} _{n=0}^{\infty}$
converges weakly to $v$, which is a common fixed point of the family
$\left\{ T_{i}\right\} _{i=1}^{m}$.
\end{proof}

\section{\label{sec5}Proofs of the main results}

\makeatletter \renewenvironment{proof}[1][\proofname\space of Theorem \ref{thm1}] {\par\pushQED{\qed}\normalfont\topsep6\p@\@plus6\p@\relax\trivlist\item[\hskip\labelsep\bfseries#1\@addpunct{.}]\ignorespaces}{\popQED\endtrivlist\@endpefalse} \makeatother
\begin{proof}
\textit{(i)} By Lemma \ref{lem2}, we have
\begin{equation}
\cap_{i=1}^{m}C_{i}\subset\cap_{n=0}^{\infty}\mathrm{Fix}(S_{n}).\label{eq:-11}
\end{equation}
Let $n\in\mathbb{N}$. By Lemma \ref{lem1}, $S_{n}$ is strongly
nonexpansive and hence nonexpansive. Since $\mathrm{Fix}(S)\subset\cap_{i=1}^{m}C_{i}$,
it follows that $\cap_{i=1}^{m}C_{i}\not=\emptyset$. Relation \eqref{eq:-11}
and Remark \ref{NE is QNE} imply that $S_{n}$ is quasi-nonexpansive.
Applying Theorem \ref{GenOpial}, we arrive at the result.\\
 \textit{(ii)} By Lemma \ref{lem2}, Lemma \ref{lem1} and Corollary
\ref{compos.averaged},
\begin{equation}
\mathrm{Fix}(Q)=\cap_{i=1}^{m}C_{i}\not=\emptyset.\label{eq:-12}
\end{equation}
By Lemma \ref{lem1}, $Q$ is strongly nonexpansive and therefore
nonexpansive. By Lemma \ref{Strongly nonexpansive  is asymptotically regular},
it is asymptotically regular. The desired result now follows from \eqref{eq:-12}
and Theorem \ref{Opial}.
\end{proof}
\begin{remark}
In light of Lemma \ref{lem2}, part \textit{(ii)} of \color{black}Theorem \ref{thm1} \color{black} can also be proved by using Theorem \ref{thm6 in DR} (without using Theorem
\ref{Opial}): define a quasi-periodic mapping $h$ from $\mathbb{N}$
onto $\left\{ 1,\dots,j_{f}+1\right\} $ by $h\left(n\right):=n\,\,\mathrm{mod\,}\,\left(j_{f}+1\right)+1$
for each $n\in\mathbb{N}$ and a sequence of operators $\left\{ T_{i}\right\} _{i=1}^{j_{f}+1}$
by $T_{i}:=S_{i-1}$ for each $i\in\left\{ 1,\dots,j_{f}+1\right\} $.
\end{remark}
\makeatletter \renewenvironment{proof}[1][\proofname\space of Theorem \ref{thm2}] {\par\pushQED{\qed}\normalfont\topsep6\p@\@plus6\p@\relax\trivlist\item[\hskip\labelsep\bfseries#1\@addpunct{.}]\ignorespaces}{\popQED\endtrivlist\@endpefalse} \makeatother
\begin{proof}
Let $F:=\left\{ S_{n}\mid \,n\in\left\{ 0,\dots,M-1\right\} \right\} $.
Since the sequence\\
 $\left\{ \left\{ C_{m,r}\left(n\right)\left(j\right)\right\} _{j=1}^{r}\right\} _{n=0}^{\infty}$
is $M$-quasi-periodic, it is clear that $F$ contains all possible
operators $S_{n}$ \color{black}(for all $n\in\mathbb{N}$)\color{black}. Set $T:\left\{ 1,\dots,p\right\} \rightarrow F$
to be the bijection, where $p$ is the cardinality of the set $F$.
Define $h:\mathbb{N}\rightarrow\left\{ 1,\dots,p\right\} $ by $h\left(n\right)=T^{-1}\left(S_{n}\right)$.
Clearly, $h$ is a quasi-periodic mapping onto $\left\{ 1,\dots,p\right\} $.
Then $S_{n}=T_{h\left(n\right)}$ for each natural number $n$. Now
\textit{(i)} follows from Lemma \ref{lem1} and Theorem \ref{thm6 in DR},
while \textit{(ii)} follows from \textit{(i)} and Lemma \ref{lem2}.
\end{proof}
\makeatletter \renewenvironment{proof}[1][\proofname\space of Theorem \ref{thm3}] {\par\pushQED{\qed}\normalfont\topsep6\p@\@plus6\p@\relax\trivlist\item[\hskip\labelsep\bfseries#1\@addpunct{.}]\ignorespaces}{\popQED\endtrivlist\@endpefalse} \makeatother
\begin{proof}
 \color{black}Since $\cap_{j=1}^{k}\mathrm{Fix}(T_{j})\not=\emptyset$, by Theorem \ref{thm6 in DR}, the sequence $\left\{ P_{n}x\right\} _{n=0}^{\infty}$ converges weakly to a point $x_{*}\in\cap_{j=1}^{k}\mathrm{Fix}(T_{j})$. Since $\cap_{i=1}^{m}C_{i}$ has an interior point, by Lemma \ref{lem2}, $\cap_{n=1}^{j_{f}}\mathrm{Fix}(S_{n})=\cap_{i=1}^{m}C_{i}\not=\emptyset$. By Lemma \ref{lem1} and Corollary \ref{compos.averaged}, $\mathrm{Fix}(Q)=\cap_{n=1}^{j_{f}}\mathrm{Fix}(S_{n})$. As a result, in both cases \it{(i)} and \it{(ii)}
\[
x_{*}\in\cap_{j=1}^{k}\mathrm{Fix}(T_{j})\subset\mathrm{Fix}(Q)=\cap_{n=1}^{j_{f}}\mathrm{Fix}(S_{n})=\cap_{i=1}^{m}C_{i}.
\]
 \end{proof}
 \color{black}

\section*{Funding}
Simeon Reich was partially supported by the Israel Science Foundation (Grant 820/17), the Fund for the Promotion of Research at the Technion
and by the Technion General Research Fund.

\section*{Acknowledgments}
This work is accepted for publication in the journal Optimization Methods and Software, November 2022.

\end{document}